\documentclass[12 pt]{amsart}

\usepackage{hyperref}
\usepackage{etex}
\usepackage[shortlabels]{enumitem}
\usepackage{amsmath}
\usepackage{amsxtra}
\usepackage{amscd}
\usepackage{amsthm}
\usepackage{amsfonts}
\usepackage{amssymb}
\usepackage{eucal}
\usepackage[all]{xy}
\usepackage{graphicx}
\usepackage{tikz-cd}
\usepackage{mathrsfs}
\usepackage{subfiles}
\usepackage{mathpazo}
\usepackage[colorinlistoftodos, textsize=tiny]{todonotes}
\usepackage{morefloats}
\usepackage{pdfpages}
\usepackage{thm-restate}
\usepackage[utf8]{inputenc}
\usepackage{epigraph}

\graphicspath{ {images/} }

\RequirePackage{color}
\definecolor{myred}{rgb}{0.75,0,0}
\definecolor{mygreen}{rgb}{0,0.5,0}
\definecolor{myblue}{rgb}{0,0,0.65}

\usepackage{hyperref}
\hypersetup{citecolor=blue}
\usepackage{tikz}
\usetikzlibrary{matrix,arrows,decorations.pathmorphing}

\theoremstyle{plain}
\newtheorem{theorem}[subsection]{Theorem}

\newtheorem{proposition}[subsection]{Proposition}
\newtheorem{lemma}[subsection]{Lemma}

\theoremstyle{definition}
\newtheorem{definition}[subsection]{Definition}
\newtheorem{remark}[subsection]{Remark}

\theoremstyle{remark}
\newtheorem{notation}[subsection]{Notation}

\numberwithin{equation}{section}

\newcommand\nc{\newcommand}
\nc\on{\operatorname}
\nc\renc{\renewcommand}

\newcommand\bp{{\mathbb P}}

\newcommand\bz{{\mathbb Z}}

\newcommand\sce{\mathscr E}
\newcommand\scf{\mathscr F}
\newcommand\scg{\mathscr G}
\newcommand\sch{\mathscr H}
\newcommand\sci{\mathscr I}

\newcommand\sck{\mathscr K}
\newcommand\scl{\mathscr L}
\newcommand\scm{\mathscr M}
\newcommand\scn{\mathscr N}
\newcommand\sco{\mathscr O}
\newcommand\scp{\mathscr P}
\newcommand\scq{\mathscr Q}

\newcommand\scv{\mathscr V}

\newcommand \ra{\rightarrow}

\DeclareMathOperator\spec{\text{Spec }}

\newcommand \mg{{\mathscr M_g}}

\DeclareMathOperator\supp{Supp}

\DeclareMathOperator\pgl{PGL}

\DeclareMathOperator\isom{Isom}
\DeclareMathOperator\locfreelocus{\scf}

\def\listtodoname{List of Todos}
\def\listoftodos{\@starttoc{tdo}\listtodoname}

\title{The locus of plane curves in the moduli stack of curves}
\subjclass[2020]{Primary 14H10; Secondary 14H50}
\keywords{The moduli stack of curves, plane curves, valuative criteria,
Brill-Noether theory}
\author{Aaron Landesman}

\begin{document}

\maketitle

\begin{abstract}
	Let $d \geq 4$ and let
	$U_d$ denote the locus of smooth curves
	in the Hilbert
	scheme of degree $d$ plane curves.
	If the members of $U_d$ have genus $g$,
	let $\mathscr{M}_g$ denote the moduli stack of genus $g$ curves.
	We show that
	the natural map $[U_d/\operatorname{PGL}_3] \to \mathscr{M}_g$
	is a locally closed embedding.
\end{abstract}

\section{Introduction}

One of my most substantial
interactions with the work of Fedor Bogomolov was through his paper with Tschinkel
``Unramified Correspondences''
\cite{bogomolovT:unramified-correspondences}.
This concerns the construction of a family of covers of curves where
the Jacobian of each source curve has a copy of a fixed elliptic curve.
Since I learned about the construction in 
\cite{bogomolovT:unramified-correspondences}, it has been one of my favorite
geometric examples, which I frequently have the pleasure of describing to my
colleagues.
This example led to a counterexample to the Putman-Wieland
conjecture, regarding maps of families of curves, given in
\cite{markovic}. The construction also led to counterexamples to Prill's
problem, given in
\cite{landesmanL:prill}. 

In this paper, we will study some foundational properties of a different map
between two 
moduli spaces of curves, which also has connections to several works of
Bogomolov.
As we will explain below, the moduli stack of plane curves can be presented as
the quotient of an open in projective space, modulo the action of $\pgl_3$.
There have been a number of important papers studying whether this stack is
rational, see
\cite{shepherd1988rationality} and \cite{bohning2010rationality}.
Indeed, these rationality results are in the same spirit as many works of
Bogomolov, proving (stable) rationality of quotients of projective space by a group
action, see 
\cite{bogomolov:stable-rationality}
\cite{bogomolov:rationality-of-hyperelliptic},
and
\cite{bogomolovK:rationality}.
In this paper, we will study the map from the moduli stack of plane curves to the
moduli stack of curves.

Our main goal in this paper is to show that various natural definitions of the stack of
plane curves agree.
On the one hand, one might wish to define it as the substack of $\mg$, the
moduli stack of smooth proper geometrically connected genus $g$ curves, whose
geometric points are plane curves.
On the other hand,
there is a Hilbert scheme 
of degree $d$ curves in $\mathbb P^2$. Let $U_d$ be the smooth locus of this
Hilbert scheme. This has
an action of $\pgl_3 \simeq \on{Aut}(\mathbb P^2)$.
Then, one might wish to define the stack of smooth plane curves to be $[U_d/\pgl_3]$.
One can use the universal curve over $U_d$ to construct a map $[U_d/\pgl_3] \to
\mg$, with $g = \binom{d-1}{2}$.
Our main result is that the two candidate definitions above actually agree.

\begin{theorem}
	\label{theorem:plane-intro}
	For $d \geq 4$ and $g = \binom{d-1}{2}$, the map $[U_d/\pgl_3] \to \mg$ is a locally closed
	embedding of stacks.
\end{theorem}

We prove this by defining an intermediate stack
$\scp_d$ and
first identifying $[U_d/\pgl_3]$ with $\scp_d$ in
\autoref{proposition:plane-curves-to-pd}.
We then show the natural map $\scp_d \to \mg$ is a locally closed embedding in
\autoref{theorem:locally-closed-embedding}.

\begin{remark}
	\label{remark:}
	In particular, we obtain that the 
	geometric points of $\mg$ corresponding to plane curves underlie a smooth
	locally closed substack of $\mg$ of dimension $\binom{d+2}{2}-9$.
\end{remark}
\begin{remark}
	\label{remark:}
The reason we restrict to the case $d \geq 4$, is that, when $d < 4$, plane curves of degree
$d$ will have genus at most $1$, so the moduli stack of curves $\mg$ will not be
Deligne-Mumford.
\end{remark}

\subsection{Idea of the proof}
We use $\scp_d$ to denote the moduli stack of degree $d$ plane curves, whose $T$
points parameterize smooth proper curves $C \to T$ with geometrically connected fibers
of genus $\binom{d-1}{2}$ and a closed embedding $C \to P$ into a
$2$-dimensional Brauer-Severi
scheme $P$ over $T$.
The idea for our proof that $\scp_d \to \mg$ is a locally closed embedding is to first show that $\scp_d$ can be identified as an open
subfunctor of $\scg^2_d$, which parameterizes curves together with a $g^2_d$.
Known facts about $\scg^2_d$ then imply $\scp_d \to \mg$ is representable by
schemes.
To prove $\scp_d \to \mg$ is a locally closed embedding, we can then use
Mochizuki's valuative criterion for a map to be a locally closed embedding,
which reduces us to verifying the above map is a monomorphism and satisfies a
certain valuative criterion. The main difficult is to verify it is a
monomorphism.
We do this using classical
facts about $g^2_d$'s over algebraically closed fields. Some care is needed to
deal with the case of positive characteristic.
We then verify the valuative criterion via a
moduli theoretic argument.

\subsection{Outline}
The outline of this paper is as follows.
We first prove some classical facts about line bundles on
plane curves in \autoref{section:classical-facts}.
Next, we define the stack of plane curves in \autoref{section:moduli-of-plane-curves}
and relate it to two other stacks parameterizing plane curves,
$[U_d/\pgl_3]$ and $\scg^2_d(p)$.
Finally, we show that the natural map $\scp_d \ra \mg$
is a locally closed embedding in \autoref{section:locally-closed-embedding}.

\begin{remark}
	\label{remark:}
	The results of this article have already been used in a number of
	places. They have been cited in 
\cite{canningL:on-the-chow-and-cohomology-rings},
studying the Chow rings of moduli
spaces of curves,
and also in 
\cite{janbaziS:finiteness},
studying arithmetic statistics questions related to plane
curves.
The results of this article have also been used any time 
one wants to view the locus of plane curves as
a substack of the moduli
stack of curves.
There are a number of papers calling $[U_d/\pgl_3]$ the moduli stack of
plane curves. See, for example
\cite[p.\ 51]{shepherd1988rationality}, \cite[p.\ 1]{hacking2004compact}, and \cite{bohning2010rationality}.
The main goal of this paper is to identify this moduli stack with the locus in
$\mg$ parameterizing plane curves.
\end{remark}

All stacks, unless otherwise specified, are defined
over $\spec \bz$.

\section{Classical Facts about plane curves}
\label{section:classical-facts}

In this section, we recall some classical facts regarding plane curves over an algebraically closed field.
The main result of this section is \autoref{proposition:reduced-point}, which states that
a smooth plane curve has a unique $g^2_d$, and that $g^2_d$
corresponds to a reduced point in the scheme parameterizing $g^2_d$'s.
We will need this later to test a certain map of stacks is an isomorphism, by testing it on geometric points.

Many of the results of this section can be found in the exercises
\cite[Appendix A, Exercises 17 and 18]{ACGH:I}, and we include proofs for completeness.
We note that the results of this section hold over fields
of arbitrary characteristic (as we prove) even though
\cite[Appendix A, Exercises 17 and 18]{ACGH:I} typically
has the hypothesis that the field has characteristic $0$.
This independence of characteristic is crucial for defining
our stacks over $\spec \bz$ (instead of
over a field of characteristic $0$).
	
We begin with some standard definitions.
\begin{definition}
	\label{definition:}
	Let $k$ be an algebraically closed field.
	A $0$-dimensional subscheme $S \subset \bp^2_k$ of degree $d$ is said to {\bf impose independent conditions
	on curves of degree $n$} if
		\begin{align*}
			h^0(\bp^2_k, \sci_S(n)) = h^0(\bp^2_k, \sco_{\bp^2_k}(n)) - d,
		\end{align*}
		where $\sci_S \subset \sco_{\bp^2_k}$ is the ideal sheaf of $S$.
\end{definition}
\begin{definition}
	\label{definition:}
	A $g^r_d$ on a smooth curve $C$ is a line bundle $\scl$ of degree $d$ on $C$ with $h^0(C, \scl) \geq r+1$.
\end{definition}

\subsection{Showing there is a single $g^2_d$}

Our first goal is to show that a smooth plane curve
has only one $g^2_d$ in \autoref{proposition:unique-g2d}.

We start with a lemma characterizing when finite reduced subschemes
of $\bp^2$
supported on at most $n+2$ points
impose independent conditions on curves of degree $n$.
\begin{lemma}
	\label{lemma:independent-conditions}
	Let $S$ be any reduced subscheme of $\bp^2$ whose support
	consists of $d$ points, with $d \leq n+1$.
	Then, $S$ imposes independent conditions on curves of degree $n$.
	Further, if $S \subset \bp^2$ is a reduced subscheme supported on $n+2$ points,
	then $S$ fails to impose independent conditions on curves of degree $n$ if and only if $S$ is contained
	in some line.
\end{lemma}
\begin{proof}
	Since $S$ has degree $d$, it follows from the exact sequence
	\begin{equation}
		\label{equation:}
		\begin{tikzcd}
			\sci_S(n) \ar {r} & \sco_{\bp^2_k}(n) \ar {r} & \sco_S 
		\end{tikzcd}\end{equation}
	that
	\begin{align*}
			h^0(\bp^2_k, \sci_S(n)) \geq h^0(\bp^2_k, \sco_{\bp^2_k}(n)) - d.
	\end{align*}

	To conclude the proof,
	by induction on the degree of $S$, it suffices to show that for any $d \leq n+1$
	we can find some plane curve of degree $n$ passing through all but one point of $S$,
	but not passing through the last point of $S$. Further in the case $d= n+2$, it suffices to show
	we can find
	such a curve passing through all but one point of $S$, provided the $n+2$ points do not lie on a line.

	In the case $d \leq n+1$,
	to see this, let $p_d$ denote a particular point of $S$ and for each point $p_i \in S, p_i \neq p_d$, choose a line $\ell_i$
	passing through $p_i$ but not through $p_d$.
	In the case $d \leq n+1$,
	taking $C$ to be the union of the lines $\cup_{i\neq d} \ell_i$ provides a curve of degree $\leq n$
	passing through all but one point of $S$. Taking the union of this with a curve of degree $n-d-1$ not passing
	through $p_d$ provides the desired curve of degree $n$.

	To conclude, we only need verify that if $S$ is supported on $n+2$ non collinear points, there is some curve passing through all
	but one of these points.
	Since the points of $S$ are not collinear, we may choose $p_1, p_2$, and $p_3$ in the support of $S$ so that $p_3$
	does not lie on the line joining $p_1$ and $p_2$.
	It suffices to show
	there is a curve passing through all points except $p_3$.
	Then, let $\ell_1$ be the line joining $p_1$ and $p_2$.
	For $2 \leq i \leq n$, let $\ell_i$ be a line passing through $p_{i+2}$ but not $p_3$.
	Then, $\cup_{i=1}^n \ell_i$ provides the desired curve of degree $n$ not passing through $p_3$.
\end{proof}

Using \autoref{lemma:independent-conditions}
we can compute the cohomology of invertible sheaves of low degree
on smooth plane curves.
\begin{lemma}
	\label{lemma:low-degree-g1m}
	Let $C$ be a smooth plane curve of degree $d$ over an algebraically closed field $k$.
	Let $p_1, \ldots, p_m$ be distinct points and $\scl := \sco_C(p_1+ \cdots+ p_m)$.
	Then, if $m \leq d-2$, we have $H^0(C, \scl) = 1$.
	If $m = d - 1$, then $h^0(C, \scl) =1$ unless $p_1, \ldots, p_m$ lie in a line $\ell$, in which case
	$h^0(C, \scl) = 2$ and $h^0(C, \sco_C(\ell \cap C)) = 3$. 
\end{lemma}
\begin{proof}
	Let $S := \cup_{i=1}^m p_i$.
	By \autoref{lemma:independent-conditions} applied in the case that the
	parameter called $n$ there is $d - 3$ here,
	we have
	an exact sequence
\begin{equation}
	\nonumber
\begin{tikzpicture}[baseline= (a).base]
\node[scale=.7] (a) at (0,0){
		\begin{tikzcd}
			0 \ar {r} &  H^0(\bp^2_k, \sco_{\bp^2_k}(d-3) \otimes \sci_S) \ar {r} & H^0(\bp^2_k, \sco_{\bp^2_k}(d-3)) \ar {r} & H^0(\bp^2_k, \sco_S) \ar {r} & 0.
		\end{tikzcd}
			};
\end{tikzpicture}
\end{equation}

We obtain a corresponding map of exact sequences
\begin{equation}
		\label{equation:points-to-curve-cohomology}
\begin{tikzpicture}[baseline= (a).base]
\node[scale=.8] (a) at (0,0){
		\begin{tikzcd}
			0 \ar {r} &  H^0(\bp^2_k, \sco_{\bp^2_k}(d-3) \otimes
			\sci_S) \ar {r}{r} \ar {d}{r} & H^0(\bp^2_k,
			\sco_{\bp^2_k}(d-3)) \ar {r} \ar {d}{r_{d-3}} &
		H^0(\bp^2_k, \sco_S) \ar {d}{r_S} \\
			0 \ar {r} &  H^0(C, \sco_{C}(d-3) \otimes \scl^\vee)\ar {r} & H^0(C, \sco_C(d-3)) \ar {r} & H^0(C, \sco_S)
		\end{tikzcd}
		};
\end{tikzpicture}
\end{equation}
	coming from the natural restriction of sheaves to $C$ and using that $\sci_S|_C \simeq \scl^\vee$.

	We will next show the map $r$ above is an isomorphism.
	Observe that the restriction maps $r_{d-3}$ and $r_S$ of \eqref{equation:points-to-curve-cohomology}
	are isomorphisms,
	as follows from the exact sequence on cohomology
	associated to 
	\begin{equation}
		\label{equation:}
		\begin{tikzcd}
			0 \ar {r} &  \sco_{\bp^2_k}(\alpha-d) \ar {r} & \sco_{\bp^2_k}(\alpha) \ar {r} & \sco_C(\alpha) \ar {r} & 0 
		\end{tikzcd}\end{equation}
	applied in the cases $\alpha = 0$ and $\alpha = d-3$.
	Therefore, the first vertical map $r$
	of \eqref{equation:points-to-curve-cohomology}
	is also an isomorphism by the five lemma.

	We next claim that $h^0(C, \scl) = 1$ either if $\deg S \leq d-2$ or
	if $\deg S = d - 1$ and $S$ is not contained
	in a line.
	Indeed, in either such case,
	by \autoref{lemma:independent-conditions}
	and the isomorphisms from
	\eqref{equation:points-to-curve-cohomology},	
	\begin{align*}
	h^0(C, \sco_C(d-3) \otimes \scl^\vee) &=
	h^0(C, \sco_{\mathbb P^2_k}(d-3) \otimes \sci_S) \\
	&=
	h^0(\mathbb P^2_k,\sco_{\mathbb P^2_k}(d-3)) - m
	\\
	&=
h^0(C,\sco_C(d-3)) - m.
	\end{align*}
	
	Next, observe that the genus of $C$ is $\frac{(d-1)(d-2)}{2}$ and so the canonical bundle of the plane curve $C$ 
	is 
$\sco_C(d-3)$,
	since $\sco_C(d-3)$ has degree $2g - 2$ and a $g$ dimensional space of global
	sections.
	Using this and geometric Riemann-Roch, we find
	\begin{align*}
		h^0(C, \scl) &= h^0(C, K_C \otimes \scl^\vee) + m - g + 1 \\
		&= h^0(C, \sco_C(d-3) \otimes \scl^\vee) + m - g + 1\\
		&= h^0(C, \sco_C(d-3)) - m + m -g +1 \\
		&= g - g + 1 \\
		&= 1.
	\end{align*}
	
	To conclude, we prove the second statement of the
	lemma.
	If $m = d-1$ and $p_1, \ldots, p_m$ are collinear, we know $h^0(C, \sco_C(p_1 + \cdots + p_{m-1})) = 1$ 
	by the above.
	Hence, $h^0(C, \sco_C(p_1 + \cdots + p_{m})) \leq 2$.
	But, if the points lie on a line $\ell$, then we know $h^0(C, C \cap
	\ell) \geq 3$, as $C$ is a plane curve.
	Since $C \cap \ell - (p_1 + \cdots + p_m)$ is an effective divisor of
	degree $1$, 
$h^0(C, \sco_C(p_1 + \cdots + p_{m}))$ and 
$h^0(C, C \cap\ell)$ differ by at most one, so 
the two inequalities must be equalities. That is,
	$h^0(C, \sco_C(p_1 + \cdots + p_{m})) = 2$ and $h^0(C, C \cap \ell) = 3$.
\end{proof}

Using the prior cohomological calculations,
in preparation for proving \autoref{proposition:unique-g2d},
we show smooth plane curves have no $g^1_{d-2}$'s and characterize
the $g^1_{d-1}$'s.
\begin{lemma}
	\label{lemma:no-g-1-m}
	Let $C \subset \bp^2_k$ be a smooth plane curve of degree $d \geq 4$,
	with $k$ an algebraically closed field. 
	Then, 
	\begin{enumerate}
		\item 
	$C$ has no $g^1_m$ for $m \leq d - 2$ and
\item 
	any $g^1_{d-1}$ is of the form $D - p$ for $p \in C$ and $D$ in the linear system $H^0(C, \sco_C(1))$.
	\end{enumerate}
\end{lemma}
\begin{proof}
	We first show that $C$ has no $g^1_m$ for $m \leq d - 2$.
	Suppose that $C$ has a $g^1_m$ for $m \leq d - 2$. Such a line bundle determines a map $C \ra \bp^1$ of degree
	at most $m$ (after possibly removing basepoints by twisting the line bundle down).
	Therefore, it suffices to show that $C$ has no dominant degree $m$ maps to $\bp^1_k$ for $m \leq d - 2$.
	That is, it suffices to show $C$ has no basepoint free $2$-dimensional linear systems of degree $m$ for $m \leq d -2$.
	
	So, suppose $C$ has some basepoint free linear systems of degree $m \leq d -2$ corresponding to a dominant map
	$C \ra \bp^1_k$.
	We next reduce to showing that $C$ has no line bundles corresponding to generically separable dominant maps
	$C \ra \bp^1_k$.
	This is automatic if $k$ has characteristic $0$.
	If $k$ has characteristic $p$,
	and $C \ra \bp^1_k$ is not generically separable, $C
	\ra \bp^1_k$ factors as the composite of
	$C \to C^{(p^n)} \to \mathbb P^1$ where 
$C \to C^{(p^n)}$ is relative $n$-fold Frobenius of $C$ and 
$C^{(p^n)} \to \mathbb P^1$ is dominant and generically separable
\cite[\href{https://stacks.math.columbia.edu/tag/0CD2}{Tag
0CD2}]{stacks-project}.
If $C$ is a plane curve, there is an embedding $C \to \mathbb P^2_k$ and so
pulling back along the Frobenius morphism $\on{Frob}_{p^n} : \spec k \to \spec
k$ induces an embedding $C^{(p^n)} \to (\mathbb P^2_k)^{(p^n)} \simeq \mathbb
P^2_k$.
If we show $C^{(p^n)}$ has no $g^1_m$ for $m \leq d -2$, this will imply the
hypothetical map $C^{(p^n)} \to \mathbb P^1_k$ does not exist. Hence, replacing
$C$ with $C^{(p^n)}$, we can
assume $C \to \bp^1_k$ is generically separable.

	So, we now show that there are no generically separable maps $C \ra \bp^1_k$.
	Since this map is generically separable, a generic fiber of the map will
	be reduced, hence of the form $p_1 + \cdots + p_m$ with $p_1, \ldots,
	p_m$ distinct.
	Thus, it remains to show there are no line bundles $\scl = \sco_C(p_1+ \cdots+ p_m)$ with
	$h^0(C, \scl) \geq 2$. 
	This holds by \autoref{lemma:low-degree-g1m},
	so $C$ has no $g^1_m$'s.

	We next verify the second claim that only $g^1_{d-1}$'s on $C$ are given by divisors of the form $D-p$ with $p \in C$ and $D$ in the linear system 
	$H^0(C, \sco_C(1))$.
	The proof is similar to the previous case.
	Let $\scl$ be some $g^1_{d-1}$.
	Note that $\scl$ must be basepoint free, as otherwise, twisting down by the basepoints, we would obtain
	some $g^1_m$ for $m \leq d-2$, contradicting the previous part.
	Thus $\scl$ determines a map $C \ra \bp^1_k$.
	Similarly to the argument made above, we can factor this as the
	composition of a power of relative Frobenius and a generically separable
	map. Since the relative Frobenius map has degree $p>1$, it suffices to
	verify the claim in the case that $C \to \bp^1_k$ is
	generically separable.
	Hence, we may assume $\scl$ determines a generically separable dominant morphism $C \ra \bp^1_k$.
	Therefore, we may assume that $\scl \cong \sco_C(p_1 + \cdots + p_{d-1})$ for $p_1, \ldots, p_{d-1}$ distinct.
	Then, by \autoref{lemma:low-degree-g1m},
	we have $h^0(C, \scl) = 1$ unless the points $p_1, \ldots, p_{d-1}$ all lie on a line.
	In the case that the points $p_1, \ldots, p_{d-1}$ lie on a line, taking $D$ to be the intersection
	of $C$ with that line in $\bp^2_k$, we obtain that $p_1 + \cdots + p_{d-1} = D - q$,
	where $q$ is by definition $D - (p_1 + \cdots + p_{d-1})$.
\end{proof}

Using \autoref{lemma:no-g-1-m}, we can now deduce the plane curves have a unique
$g^2_d$.
\begin{proposition}
	\label{proposition:unique-g2d}
	Let $C \subset \bp^2_k$ be a smooth plane curve of degree $d \geq 4$,
	with $k$ an algebraically closed field. 
	Then, C has at most one $g^2_d$ and that $g^2_d$ is a complete linear series.
\end{proposition}
\begin{proof}
	First, suppose $C$ is a smooth curve with a $g^2_d$. That is, $C$ has an invertible sheaf $\scl$
	with $h^0(C, \scl) \geq 3$.
	We claim $h^0(C, \scl) = 3$.
	To see this, let $D$ be an effective divisor in the linear system $H^0(C, \scl)$.
	Let $p_1$ and $p_2$ be two general points of $C$.
	If $h^0(C, \scl) > 3$, then $H^0(C, \mathscr O(D - p_1 - p_2)) \geq 2$
	and hence after twisting down $D - p_1 - p_2$ by additional points if $D
	- p_1 - p_2$ is not already a $g^1_{d-2}$, we
	obtain a $g^1_{m}$ on $C$ with $m \leq d - 2$, which does not exist by
	\autoref{lemma:no-g-1-m}.
	
	To complete the proof, we only need to check $C$ has a unique $g^2_d$.
	For this, let $\scm$ be any $g^2_d$.
	Then, for any $D$ in the linear system $H^0(C, \scm)$ and $p \in \supp(D)$, we have
	$\sco_C(D - p)$ is a $g^1_{d-1}$.
	By \autoref{lemma:no-g-1-m},
	$D - p$ consists of $d-1$ collinear points. Therefore, there is some point $q$ so that
	$\scm(-p) \cong \sco_C(1)(-q)$.
	Consider the invertible sheaf,
	\begin{align*}
	\scn := \scm \otimes \sco_C(1)^\vee \cong \sco_C(-q+p).
	\end{align*}

	To conclude the proof, it suffices to show $p = q$.
	We know $\scm^{\otimes (d-3)} \cong K_C$,
	since $\scm^{\otimes (d-3)}$ and $K_C$ are degree $2g-2$ bundles with a
	$g$ dimensional space of global sections.
	Therefore, $\scn^{\otimes (d-3)} \cong \sco_C$.
	This implies $\sco_C((d-3)p) \otimes \sco_C(-(d-3)q) \cong \sco_C$.
	Using \autoref{lemma:no-g-1-m}, we see $H^0(C, \sco_C((d-3)p)) = 1$.
	This implies that the only section of $\sco_C((d-3)p)$ is the trivial section,
	which does not vanish at any point other than $p$,
	and so $\sco_C((d-3)p) \otimes \sco_C(-(d-3)q)$ has no sections unless $p = q$.
	Therefore, $\scm \cong \sco_C(1)$, as desired.
\end{proof}

\subsection{Showing that $g^2_d$ is reduced}

Our next goal is to show that the unique $g^2_d$
on a smooth plane curve (whose uniqueness was established
in \autoref{proposition:unique-g2d}) corresponds to
a reduced point of a parameter space for $g^2_d$'s
(which we shall define in \autoref{definition:grd}). We prove this in
\autoref{proposition:reduced-point}.
In order to do so, we first recall a standard fact that
plane curves are projectively normal.

\begin{lemma}
	\label{lemma:projective-normality}
	Let $k$ be a field. Any smooth plane curve $C \subset \bp^2_k$ is projectively normal, meaning
	that for all $n > 0$, the map $H^0(\bp^2_k, \sco_{\bp^2_k}(n)) \ra H^0(C, \sco_C(n))$ is surjective.
\end{lemma}
\begin{proof}
	By the long exact sequence associated to
	\begin{equation}
		\label{equation:}
		\begin{tikzcd}
			0 \ar {r} &  \sci_C \ar {r} & \sco_{\bp^2_k} \ar {r} & \sco_C \ar {r} & 0 
		\end{tikzcd}\end{equation}
	to verify projective normality, we only need verify $H^1(\bp^2, \sci_C(n)) = 0$ for all $n \geq 0$.
	Say $C$ has degree $d$. Then,
	$\sci_C(n) \cong \sco_{\bp^2_k}(n-d)$.
	Therefore, $H^1(\bp^2, \sci_C(n)) = H^1(\bp^2, \sco_{\bp^2_k}(n-d)) = 0$.
\end{proof}

We next define the scheme $\scg^r_d(p)$ for $p: C \ra S$
a smooth proper curve with geometrically connected fibers of genus $g$.

\begin{definition}[~\protect{\cite[Chapter XXI, Definition 3.12]{ACMG:geometryOfCurves}}]
	\label{definition:grd}
	Suppose $p : C \ra S$ is a smooth proper curve with geometrically
	connected fibers of genus $g$.
	Define the fibered category
	$\scg^r_d(p)$ sending a map $f:T \ra S$
	to the set of equivalence classes of pairs $\left( \scl, \sch \right)/\sim$ defined as follows:
	Let $\iota: t \ra T$ be a point and define the corresponding fiber square
	\begin{equation}
		\label{equation:}
		\begin{tikzcd} 
			C_t \ar {r}{\iota_C} \ar {d}{p_t} & C_T \ar {d}{p_T} \\
			t \ar {r}{\iota} & T.
		\end{tikzcd}\end{equation}
	A pair $(\scl, \sch)$ consists of a line bundle $\scl$
	on $C \times_S T$ whose restriction to each fiber of $p_T: C_T \ra T$
	has degree $d$ and a locally free rank $r+1$ subsheaf $\sch \subset p_{T*}\scl$ so that for
	each fiber $\iota: t \ra T$,
	the natural composition 
	\begin{align}
		\label{equation:base-change-restriction}
		\iota^* \sch \ra \iota^* p_{T*} \scl \ra p_{t*} \iota_C^* \scl
	\end{align}
	is injective.
	The equivalence relation $\sim$ defined on pairs $\left( \scl, \sch \right)$ dictates that two pairs
	$\left( \scl, \sch \right)$ and $\left( \scl', \sch' \right)$ are equivalent
	if there is an invertible sheaf $\scq$ on $T$ and an isomorphism $\scl' \cong \scl \otimes p_T^* \scq$
	which induces an isomorphism $\sch' \simeq \sch \otimes \scq$.
\end{definition}
\begin{theorem}[~\protect{\cite[Chapter XXI, Theorem 3.13]{ACMG:geometryOfCurves}}]
	\label{theorem:representability-of-grd}
	For $p: C \ra S$ a smooth proper curve with geometrically connected fibers
	of genus $g$ admitting a section, the functor
	$\scg^r_d$ defined in \autoref{definition:grd} is represented by an $S$ scheme.
\end{theorem}
We are now ready to state and prove the main result of this section.
\begin{proposition}
	\label{proposition:reduced-point}
	Let $C$ be a degree $d \geq 4$ smooth plane curve over an algebraically closed field $k$.
	Then, $\scg^2_d(p:C \ra \spec k)$ is isomorphic to a copy of $\spec k$.
	That is, it is a reduced point.
\end{proposition}
\begin{proof}
	Since the underlying set of $\scg^2_d(p)$ is the set
	of $g^2_d$'s on $C$, together with a $3$-dimensional
	space of global sections,
	by \autoref{proposition:unique-g2d}, $\scg^2_d(p)$ is supported on a point.
	Here we are using that the unique $g^2_d$
	is not a $g^3_d$, as was shown in \autoref{proposition:unique-g2d}.

	It only remains to show this point is reduced.
	Indeed, using 
	\cite[Chapter IV, Proposition 4.1(iii)]{ACGH:I}, 
	(whose proof holds equally well in positive characteristic,)
	the tangent space to the unique point of $\mathscr G^2_d(p)$ will be $0$-dimensional if the multiplication map
	\begin{align*}
		H^0(C, \sco_C(1)) \otimes H^0(C, K_C \otimes \sco_C(-1)) \ra H^0(C, K_C)
	\end{align*}
	is surjective.
	Under the identification $K_C \cong \sco_C(d-3)$,
	we want to show the map
	\begin{align*}
		H^0(C, \sco_C(1)) \otimes H^0(C, \sco_C(d-4)) \ra H^0(C, \sco_C(d-3))
	\end{align*}
	is surjective.

	To verify this, note that we have a commutative square
	\begin{equation}
		\label{equation:}
		\begin{tikzcd} 
			H^0(\bp^2_k, \sco_{\bp^2_k}(1)) \otimes H^0(\bp^2_k, \sco_{\bp^2_k}(d-4)) \ar {r} \ar {d} & H^0(\bp^2_k, \sco_{\bp^2_k}(d-3)) \ar {d} \\
			H^0(C, \sco_C(1)) \otimes H^0(C, \sco_C(d-4)) \ar {r} & H^0(C, \sco_C(d-3)).
		\end{tikzcd}\end{equation}
	We know that the vertical maps are surjective by projective normality of $C$, as shown in \autoref{lemma:projective-normality}.
	Further, the top horizontal map is surjective, since every degree $d-3$ polynomial is a linear combination of products
	of degree $1$ and degree $d - 3$ polynomials in $3$ variables.
	Therefore, the bottom horizontal map is also surjective, as desired.
\end{proof}

\section{The moduli stack of plane curves}
\label{section:moduli-of-plane-curves}

In \autoref{subsection:definition-of-plane-curves},
we define the moduli stack of plane curves of degree $d$, 
which we will denote $\scp_d$, and verify it is isomorphic to a quotient stack
$[U_d/\pgl_3]$ for $U_d$ the open of the Hilbert scheme parameterizing smooth
plane curves of degree $d$.
In particular, this implies that $\scp_d$ 
is a smooth algebraic stack
of finite type.
Then, in \autoref{subsection:g2d-relation} we relate $\scp_d$ to $\scg^2_g(p)$,
for $p: C \to S$ a smooth proper curve with geometrically connected fibers of
genus $g$, showing we can realize
$S \times_{\mg} \scp_d$ as an open subfunctor of $\scg^2_d(p)$.

\subsection{Defining the stack of plane curves}
\label{subsection:definition-of-plane-curves}

To begin, we define the stack of plane curves.
Recall that a Brauer-Severi scheme of dimension $n$ over a base $T$ is a scheme
$P \to T$ so that there is an \'etale cover $T' \to T$ with $P \times_T T'
\simeq \mathbb P^n_{T'}$.
\begin{definition}
	\label{definition:plane-curves}
	Let $d \geq 1$ be an integer.
	Let $g := \binom{d-1}{2}$.
	Define the {\bf moduli stack of degree $d$ plane curves} denoted $\scp_d$
	to be the fibered category of
	triples $(f, P, h)$ where 
	$f :C \ra T$ is a smooth proper curve with geometrically connected genus
	$g$ fibers,
		$P$ is a Brauer-Severi scheme of dimension $2$ over $T$ and $h: C \to P$
	a closed embedding over $T$.

	A morphism 
	$(f': C' \ra T', P', h') \ra (f:C \ra T, P, h)$ is a triple of maps
	$(T' \to T, j: C' \to T' \times_T C, i:P' \to T' \times_T P)$ so
	and $i$ and $j$ are isomorphisms and $j$ restricts to $i$ in the sense
	that $i \circ h' = h \circ j$.
	This makes $\scp_d$ into a fibered category over the category of schemes
	by sending $(f: C \to T, P,h)$ to $T$.
\end{definition}

It is not immediately obvious $\scp_d$ is a stack, but we can verify this using
an alternate definition.
\begin{notation}
	\label{notation:}
	Let $H_d$ denote the Hilbert scheme of degree $d$ plane curves. 
	Abstractly, this is isomorphic to $\mathbb P^{\binom{d+2}{2}-1}$
	corresponding to the $\binom{d+2}{2}$ coefficients of polynomials of
	degree $d$ in $3$ variables.
	Let $U_d \subset H_d$ denote the open subscheme parameterizing smooth
	plane curves of degree $d$.
	This has an action of $\on{PGL}_3 = \on{Aut}(\mathbb P^2)$ where $\alpha \in
	\on{PGL}_3(T)$ sends a plane
	curve $C \xrightarrow{h} \mathbb P^2_T$ to the composite $C
	\xrightarrow{h} \mathbb P^2_T \xrightarrow{\alpha} \mathbb P^2_T$.
	We let $[U_d/\pgl_3]$ denote the resulting quotient stack.
\end{notation}

\begin{proposition}
	\label{proposition:plane-curves-to-pd}
	For $d \geq 4$,
	there is an equivalence $\scp_d \simeq [U_d/\pgl_3]$.
\end{proposition}
\begin{proof}
	First, we describe the map $\scp_d \to [U_d/\pgl_3]$.
		To give a map $T \to [U_d/\pgl_3]$ we need to specify a $\pgl_3$ torsor
	$R$ over $T$ with a $\pgl_3$ equivariant map $R \to U_d$ using the data
	of a map $T \to \scp_d$, corresponding to $C \to P \to T$ as in 
	\autoref{definition:plane-curves}.

	Let $R$ denote the $\pgl_3$ torsor
	$\on{Isom}_T(P, \mathbb P^2_T)$ over $T$.
	This is a $\pgl_3$ torsor because the pullback of $P$ to $R$ is $\mathbb
	P^2_R$ and so one can identify $R \times_T R$ with $\on{Isom}_R(P_R,
	\mathbb P^2_R) \simeq \on{Isom}_R(\mathbb P^2_R,
	\mathbb P^2_R) \simeq \pgl_{3,R}.$
	By composing the base change of the given embedding $C \to P$ with the isomorphism $P_R \to
	\mathbb P^2_R$ over $R,$ we obtain an embedding $C_R \to P_R \to \mathbb
	P^2_R$ realizing $C_R$ as a degree $d$ plane curve (since $C$ has genus
		$g = \binom{d-1}{2}$ and is embedded as a curve in $\mathbb
	P^2$). This yields the desired $\pgl_3$ equivariant map $R \to U_d$.
	Since this map sends the automorphisms of the map $T \to \scp_d$ to
	automorphisms of the corresponding point of $[U_d/\pgl_3]$,
	this defines the map $\scp_d \to [U_d/\pgl_3]$.

	We now want to construct the inverse map.
	Recall the universal property of the Hilbert scheme $U_d$. A map $T \to
	U_d$ is the same data as a curve $C \to \mathbb P^2_T$ over $T$, flat
	over $T$, which is a smooth degree $d$ plane curve on each geometric
	fiber. Equivalently, this is a smooth, geometrically connected curve $C \to
	T$ with an embedding $C \to \mathbb P^2_T$ over $T$.
	In particular, such a curve $C \to T$ is pulled back from the universal
	family of smooth plane curves with embeddings into $\mathbb P^2_{U_d}$,
	$C_d \to \mathbb P^2_{U_d}$, over $U_d$.
	That is, $C = T \times_{U_d} C_d$.
	There are compatible actions of $\pgl_3$ on $C_d \to \mathbb P^2_{U_d}
	\to U_d$ giving us a diagram $[C_d/\pgl_3] \xrightarrow{h} \mathbb
	[\mathbb P^2_{U_d}/\pgl_3] \xrightarrow{\pi}
	[U_d/\pgl_3]$, and we call $f$ the composite map.
	Note that $f$ is representable by a smooth proper curve with
	geometrically connected fibers of genus $g$, $\pi$ is a relative Brauer
	Severi scheme, and $h$ is a closed embedding, as these may all be
	verified after pullback to the smooth cover $U_d$.

Now, pulling back these stacks along a map $T \to [U_d/\pgl_3]$ we get $C \to P
\to T$ with $C = T \times_{[U_d/\pgl_3]} [C_d/\pgl_3] $ and $P =
T\times_{[U_d/\pgl_3]}[\mathbb P^2_{U_d}/\pgl_3]$, with $C$ a smooth proper curve and
$P$ a Brauer-Severi scheme of dimension $2$ over $T$, together with a closed
embedding $C \to P$. This is precisely the data of a map $T \to \scp_d$.
Since the Brauer-Severi scheme becomes trivialized along the cover $\isom_T(P,
\mathbb P^2_T)$ this construction $[U_d/\pgl_3] \to \scp_d$ is inverse to the
previous map $\scp_d \to [U_d/\pgl_3]$, and so both are equivalences.
\end{proof}

\subsection{Relating the stack of plane curves to $\scg^2_d(p)$}
\label{subsection:g2d-relation}

Having defined $\scp_d$, we next relate it to the scheme $\scg^2_d(p)$.
Recall the definition of the stack $\scg^r_d(p)$ for $p:C \ra S$ a smooth proper
curve with geometrically connected genus $g$ fibers given in \autoref{definition:grd}
and recall that it is representable when $p$ has a section,
by \autoref{theorem:representability-of-grd}.

\begin{definition}
	\label{definition:sch-and-sck}
	First, we define $\locfreelocus^2_d(p)$ to be the subfunctor of $\scg^2_d(p)$ which associates to any $S$ scheme $T$ pairs $\left( \scl, \sch \right)/\sim$
	as in \autoref{definition:grd} with the additional condition that $p_{T*}\scl$ is locally free of rank $3$
	and $p_T^* p_{T*} \scl \ra \scl$ is surjective.
	
	Next, we define $\sck^2_d(p)$ to be the subfunctor of $\locfreelocus^2_d(p)$ which associates to any $S$ scheme $T$
	pairs $\left( \scl, \sch \right)/\sim$ as in \autoref{definition:grd} with the additional conditions that $p_{T*}\scl$ is locally
	free of rank $3$, that  and that the resulting morphism $C \ra \bp
(p_{T*}\scl)$ is a closed embedding.
\end{definition}

In the next several lemmas, we show $\sck^2_d(p) \subset \locfreelocus^2_d(p)
\subset \scg^2_d(p)$ are open embeddings.

\begin{lemma}
	\label{lemma:f-to-g}
	Suppose $p:C \ra S$ is a smooth proper curve with geometrically
	connected genus $g$ fibers with a
	section. The natural map $\locfreelocus^2_d(p) \to \scg^2_d(p)$ is an
	open embedding.
	In particular, $\locfreelocus^2_d(p)$ is representable by a scheme.
\end{lemma}
\begin{proof}
	We know $\scg^2_d(p)$ is representable by some scheme $X$ with a universal invertible sheaf $\scl_X$ on the universal
	curve $p_X: C_X \ra X$ by \autoref{theorem:representability-of-grd}.
	Note that the locus of scheme $Z$ over which a finitely presented sheaf $\mathscr F$ is locally free of a given rank
	is open, since if $p \in Z$ and the stalk $\mathscr F_p$ is a free
	$\mathscr O_{Z,p}$ module, then by spreading out there is an open
	neighborhood $U$ of $p$ over which $\mathscr F|_U$ is free.
	It follows that $\locfreelocus^2_d(p)$ is then represented by the open subscheme of $X$ on which $p_{X*}\scl_X$ is locally free of rank $3$
	and the map $p_X^* p_{X*} \scl_X \ra \scl_X$ is surjective; this locus is
		open because it is the complement of the support of the cokernel
		of $p_X^* p_{X*} \scl_X \to \scl_X$,
	and the support of a sheaf is a closed locus.
\end{proof}
\begin{lemma}
	\label{lemma:k-to-f}
	Suppose $p:C \ra S$ is a smooth proper curve with geometrically
	connected fibers of genus $g$ with a
	section.
	The natural map $\mathscr K^2_d(p) \to \locfreelocus^2_d(p)$ is an open
	embedding.
	In particular, $\mathscr K^2_d(p)$ is representable by a scheme.
\end{lemma}
\begin{proof}
	Using \autoref{lemma:f-to-g}, we can assume $\locfreelocus^2_d(p)$, is represented by some scheme $Y$. 
	We first claim there is a universal family $p_Y: C_Y \to Y$
	together with a map $C_Y \to \mathbb P(p_{Y*} \scl)$ over $Y$.
	To see this, note that because $p_{Y*}\scl$ is locally free of rank $3$, we may
	construct the relative proj $\bp (p_{Y*}\scl)$. (Since the bundle
		$p_{Y*} \scl$ has rank $3$, locally on $Y$ this can be
	identified with $\mathbb P^2$.)
	The condition that $p_Y^* p_{Y*} \scl \ra \scl$ is surjective implies that the sheaf $\scl$ is basepoint free,
	and hence we obtain a resulting morphism $C_Y \ra \bp (p_{Y*}\scl)$.

Since $\sck^2_d(p)$ is the subfunctor of $\locfreelocus^2_d(p)$ where
	the above map 
$C_Y \ra \bp (p_{Y*}\scl)$
is a closed embedding,
	to show $\sck^2_d(p) \ra \locfreelocus^2_d(p)$ is an open
	embedding, it suffices to show there is an open locus
	over which
	the map $C_Y \ra \bp(p_{Y*} \scl)$ is a closed embedding.
	This is intuitively clear because the locus on which
	it is not a closed embedding should be the one in which
	the degree of some fiber of the map is more than $1$.
	We formally codify this as follows:
	Let $\scv$ denote the cokernel of the resulting map $\sco_{\bp p_{Y*} \scl_Y} \ra \phi_* \sco_{C_Y}$.
	Let $W := \supp \scv \subset \bp (p_{Y*} \scl_Y)$ denote the support of $\scv$ and let
	$Z := \phi^{-1}(W)$.
	Then, it follows that $Z$ is a closed subscheme of $C_Y$ and $Y - p_Y(Z)$ represents the functor $\sck^2_d(p)$.
\end{proof}
The final result of this section relates $\sck^2_d(p)$ to $\scp_d$.
\begin{proposition}
	\label{proposition:fiber-product-is-scheme}
		Suppose $p:C \ra S$ is a smooth proper curve with geometrically
		connected fibers of genus $g$ with
		a section, corresponding to a map $S \ra \mg$.
		There is an equivalence $\sck^2_d(p) \simeq S \times_{\mg}
		\scp_d$.
		In particular, 
		$S \times_{\mg} \scp_d$ is a scheme.
\end{proposition}
\begin{proof}
		The map $\sck^2_d(p) \to S \times_{\mg} \scp_d$ is given by as follows: On a $T$ point of $\sck^2_d(p)$, we
	obtain a smooth proper curve with geometrically connected genus $g$
	fibers $p_T: C \to T$ and an invertible sheaf $\scl$ 
	such that the induced map $h: C \to \mathbb P ((p_T)_* \mathscr L)$
	is an embedding.
	Hence, we can associate to this the $T$ point of $S \times_{\mg} \scp_d$
	corresponding to $(p_T : C \to S, \mathbb P((p_T)_* \mathscr L), h: C
	\to \mathbb P ((p_T)_* \mathscr L)).$

	This induces a map of stacks, which we want to show is an equivalence.
	For this, it suffices to show this map is fully faithful and essentially surjective on
	$T$ points for all schemes $T$ \cite[Proposition 3.1.10]{olsson2016algebraic}.
	Full faithfulness follows from the definitions of $\scp_d$ and
	$\sck^2_d(p)$, as the morphisms of objects in both cases are given by
	the same data (isomorphisms of the projective bundle or Brauer-Severi
	scheme restricting to an isomorphism on the embedded curves).
	Hence, it only remains to check essential surjectivity. That is, given
	$C \to T$ with an embedding $C \to P$ for $P$ a Brauer-Severi scheme
	over $T$, we wish to show that $P$ is actually a trivial Brauer-Severi
	scheme. That is, we want to show $P \simeq \bp \sce$ for some invertible sheaf $\mathscr
	E$.
	This is where we use that we are assuming $C \to T$ has a section
	$\sigma: T \to C$. Composing $\sigma$ with the given map $h: C \to P$, we
	obtain a section $\sigma: T \to C \to P$, implying that $P$ is a trivial
	Brauer-Severi scheme (the section induces a line in the Brauer-Severi
		scheme inverse to $P$, implying the class of the inverse of $P$
	is trivial and so $P$ also corresponds to the trivial Brauer-Severi
scheme). Hence,
	$P \simeq \bp \sce$ for some invertible sheaf $\mathscr E$. This then
	induces an invertible sheaf $\mathscr O_{\mathscr E}(1)$ on $C$ and the
	invertible sheaf $\mathscr L$ from the definition of $\sck^2_d(p)$ can
	be taken to be $h^*(\mathscr O_{\mathscr E}(1))$.

	Lastly, the claim in the final sentence that 
$S \times_{\mg} \scp_d$ is a scheme 
holds by \autoref{lemma:k-to-f}, since since $\sck^2_d(p)$ is 
a scheme.

\end{proof}

\section{Verifying that $\scp_d \ra \mg$ is a locally closed embedding}
\label{section:locally-closed-embedding}

Let $g := \binom{d-1}{2}$.
Our next goal is to verify that the natural map $\scp_d \ra \mg$, sending a pair $(f: C \ra T, \scl) \mapsto (f: C \ra T)$ is a locally closed embedding.

We do this in several steps, completing the proof in \autoref{theorem:locally-closed-embedding}.
\begin{enumerate}
	\item We show that in the case $f:C \ra T$ has a section, the resulting
		fiber product $T \times_{\mg} \scp_d$ is in fact a scheme (we
		did this in \autoref{proposition:fiber-product-is-scheme}).
	\item We show that the map $\scp_d \ra \mg$ is a monomorphism (in \autoref{proposition:monomorphism}).
	\item We show that $\scp_d \ra \mg$ is a locally closed embedding, using
		the (little known!) valuative criterion for locally closed
		embeddings (in \autoref{theorem:locally-closed-embedding}).
\end{enumerate}

Now, since $\mg$ is Deligne-Mumford, there is an \'etale cover $T \ra \mg$.
This corresponds to a smooth proper curve with geometrically connected genus $g$
fibers $C \ra T$.
After choosing a further \'etale cover of $T$, we may assume that $p:C \ra T$ has a section.
In this case, we know
from 
\autoref{theorem:representability-of-grd}
that the functor $\scg^2_d(p)$ is representable by a scheme.
That is, the fiber product
\begin{equation}
	\label{equation:tau-definition}
	\begin{tikzcd} 
		T \times_{\mg} \scp_d \ar {r}{\tau} \ar {d} & T \ar {d} \\
		\scp_d \ar {r} & \mg
	\end{tikzcd}\end{equation}
is a scheme, using \autoref{proposition:fiber-product-is-scheme}.
Hence, in order to verify our map $\scp_d \ra \mg$ is a locally closed embedding,
we only need check $\tau$ is a locally closed embedding.
For this, we will use the following valuative criterion
for locally closed embeddings.

\begin{lemma}[Valuative criterion for locally closed embeddings, ~\protect{\cite[Chapter 1, Corollary 2.13]{mochizuki2014foundations}}]
	\label{lemma:functorial-criterion-for-locally-closed-embedding}
	Suppose $S$ is a Noetherian scheme and $X$ and $Y$ are $S$-schemes of finite type.
	A morphism $f: X \ra Y$ is a locally closed embedding if and only if $f$ is a monomorphism and the following
	condition holds:
	For all discrete valuation rings $R$ with fraction field $K$ and residue field $\kappa$,
	and all maps $g:\spec R \ra Y$
	with commutative diagrams
	\begin{equation}
		\label{equation:}
		\begin{tikzcd} 
			\spec K \ar {r} \ar {d} & X \ar {d}{f} & \spec \kappa \ar{r} \ar{d} & X \ar{d}{f} \\
			\spec R \ar {r}{g} & Y & \spec R \ar{r}{g} & Y,
		\end{tikzcd}\end{equation}
	there exists a unique morphism $h:\spec R \ra X$ making the diagrams	
	\begin{equation}
		\label{equation:}
		\begin{tikzcd} 
			\spec K \ar {r} \ar {d} & X \ar {d}{f} & \spec \kappa \ar{r} \ar{d} & X \ar{d}{f} \\
			\spec R \ar {r}{g}\ar{ur}{h} & Y & \spec R \ar{r}{g}\ar{ur}{h} & Y
		\end{tikzcd}\end{equation}
	commute.
\end{lemma}

So, we need to verify the map $\tau$ of \eqref{equation:tau-definition} is a monomorphism
and that it satisfies the valuative criterion
of \eqref{lemma:functorial-criterion-for-locally-closed-embedding}.
First, we show it is a monomorphism.

\begin{proposition}
	\label{proposition:monomorphism}
	The map $\tau$ defined in \eqref{equation:tau-definition} 
	is a monomorphism.
\end{proposition}
\begin{proof}
	Using \cite[Proposition 17.2.6]{EGAIV.4} in order to verify the natural map
$\tau$ is a monomorphism, it suffices to verify each fiber
either has degree $0$ or $1$.
This can be verified on geometric points.
So, let $\spec \overline k \ra T$ be some geometric point.
Let $C_{\overline k}$ be the corresponding curve over $\overline k$.
We wish to show the fiber over $\spec \overline k$ has degree $1$.
For this, we will check the fiber is reduced and is supported on a single point.
From the definition of the stack $\scg^2_d(p)$ with $p: C_{\overline k} \ra \spec \overline k$, which has a section
as $\overline k$ is algebraically closed, we see $\scg^2_d(p)$
parameterizes the underlying set of $g^2_d$'s on $C$.
By \autoref{proposition:reduced-point},
$\scg^2_d(p)$ is a degree one scheme.
Since $T \times_{\mg} \scp^2_d \subset \scg^2_d(p)$ is an open subscheme by
\autoref{proposition:fiber-product-is-scheme}, 
it follows that $T \times_{\mg}\scp^2_d$ also has degree at most
$1$, completing the proof.
\end{proof}

We can now prove our main theorem,
by verifying the valuative criterion for locally closed embeddings
holds.

\begin{theorem}
	\label{theorem:locally-closed-embedding}
	The map $\scp_d \ra \mg$ is a locally closed embedding.
\end{theorem}
\begin{proof}
	First, it suffices to check after pull back to an \'etale cover of $\mg$, and hence it suffices to check the map $\tau$
	defined in \eqref{equation:tau-definition} is a locally closed embedding.
	For this, by \autoref{lemma:functorial-criterion-for-locally-closed-embedding}, it suffices to verify
	the map $\tau$ is a monomorphism (which follows from
	\autoref{proposition:monomorphism})
	and
	the valuative criterion for locally closed embeddings.
To apply 
	\autoref{lemma:functorial-criterion-for-locally-closed-embedding}
we are using that $\scp_d$ and $\mg$ are both finite type over $\spec \bz$.
We note that 
	$\scp_d$ is finite type over $\mathbb Z$ by 
\autoref{proposition:plane-curves-to-pd}.	

	It only remains to verify the valuative criterion
	for $\tau$ being a locally closed embedding.
	Retaining the notation of \autoref{lemma:functorial-criterion-for-locally-closed-embedding},
	the valuative criterion
	can be rephrased in the following way:

	Let $f: C_R \to \spec R$ be a smooth proper curve
whose generic and special fibers
$C_K$ and $C_\kappa$
are geometrically connected genus $g$ plane curves of degree $d$ with $g =
\binom{d-1}{2}$.
	Assume the maps $C_K \to \bp^2_K$ and $C_\kappa \to \bp^2_\kappa$ are given by invertible sheaves $\scl_K$ and $\scl_\kappa$.
	We want to show there exists a unique invertible sheaf $\scl$ of degree $d$ on $C_R$ so that 
	$\scl|_{C_K}= \scl_K$, $\scl_{C_\kappa} = \scl_\kappa$, 
	$f_* \scl$ is locally free of rank $3$, and $f^* f_* \scl \to \scl$ induces an embedding $C_R \to \bp(f_*
\scl)$, restricting to the given embeddings of the generic and special fibers.

	We first show $\scl$ is unique.
	Note that $\scp_d$ is separated over $\spec \mathbb Z$ using the valuative criterion for
	separatedness and the
	equivalence between $\scp_d$ and $\sck^2_d$ from
	\autoref{proposition:fiber-product-is-scheme}, since two closed
	embeddings of $C_R$ into $\mathbb P^2_R$ which have the same
	restriction to the generic fiber must agree.
	Since $\mg$ is also separated over $\spec \mathbb Z$,
	we obtain that the map 
$\scp_d \ra \mg$ is separated.
Then, the valuative criterion for separatedness applied to the map $\scp_d \to
\mg$ implies that $\scl$ is unique.

	Therefore, it suffices to show there exists a morphism $\spec R \ra \scp_d$, restricting to the given maps from $\spec K$ and $\spec \kappa$, which we now construct.
	Since $R$ is regular, it follows $C_R$ is regular.
	Therefore, the natural map from Weil divisors to Cartier divisors is an isomorphism, and hence
	$\scl_K \cong \sco_{C_k}(D_K)$ for some Weil divisor $D_K \subset C_K$.
	Let $D_R$ denote the closure of $D_K$ inside $C_R$.
	Let $\scl := \sco_{C_R}(D_R)$.
	By construction, we know $\scl|_{C_K} \cong \scl_K$.

	Next, we verify $\scl|_{C_\kappa} \cong \scl_\kappa$, that $f_* \scl$ is
	locally free of rank $3$ and the formation of $f_* \scl$ commutes with
	base change on $\spec R$.
	Indeed, since $f: C_R \ra R$ is a proper morphism of Noetherian schemes and $\scl$ is flat over $R$,
	we obtain that the map $q \mapsto h^0(C_q, \scl|_{C_q})$ is upper semicontinuous.
	In particular, since $h^0(C_K, \scl|_{C_K}) = 3$,
	we obtain $h^0(C_\kappa, \scl|_{C_\kappa}) \geq 3$.
It follows from \autoref{lemma:low-degree-g1m}
	that $h^0(C_\kappa, \scl|_{C_\kappa}) = 3$ and moreover from
	\autoref{proposition:unique-g2d} that 
	$\scl|_{C_\kappa} \simeq \scl_\kappa$.
	Grauert's theorem then implies that $f_* \scl$ is locally free of rank
	$3$ and that the formation of $f_* \scl$ commutes with base change on
	$\spec R$. 

	To conclude, we wish to show $f^* f_* \scl \to \scl$ is surjective and
	induces an embedding $C \to \bp(f_* \scl)$. Since $f_* \scl$ commutes
	with base change, we can verify this on each fiber, in which case it
	holds from the assumptions that each fiber was a plane curve, with
	embedding induced by the restriction of $\scl$.
\end{proof}

\section{Acknowledgements}

I thank David Zureick-Brown for running the REU project which originally prompted this question.
I thank Maksym Fedorchuk for 
suggesting the method
of showing $\scp_d \ra \mg$ is a locally closed embedding 
by showing $\scp_d \ra \scg^2_d$ is an open embedding 
and then showing $\scg^2_d \ra \mg$ is a locally
closed embedding.
I thank Anand Patel and Joe Harris for explaining why smooth plane curves have a unique $g^2_d$ and suggesting other methods of approaching this question.
I would like to thank an anonymous referee for a number of helpful comments
improving the exposition of the article.
I thank Brian Conrad for pointing out the useful valuative criterion for being a locally closed embedding. I thank Ravi Vakil and Michael Kemeny for listening to my argument in detail. I also thank 
Tony Feng,
Ben Lim, 
Arpon Raksit, 
Zev Rosengarten,
Bogdan Zavyalov,
and Yang Zhou
for helpful discussions.
This material is based upon work supported by the 
U.S. National Science Foundation under
award No. 2449797
and the
National Science Foundation Graduate Research Fellowship Program under Grant No.
DGE-1656518.

\bibliographystyle{alpha}
\bibliography{/home/aaron/Dropbox/master}

\end{document}